\documentclass[12pt,a4paper]{amsart}
\usepackage{amssymb}
\usepackage{mathabx}
\usepackage{bbm}
\usepackage{amsthm}
\usepackage{dsfont}
\usepackage{amsmath}
\usepackage{color,graphics,srcltx}
\usepackage{bm}
\usepackage{easybmat}
\usepackage{multirow,bigdelim}
\usepackage{enumerate}
\usepackage{graphicx}
\usepackage{amsthm}
\usepackage{url}
\usepackage{tikz-cd}
\usepackage{hyperref}
\hypersetup{colorlinks=true,allcolors=blue}
\usepackage{orcidlink}
\usepackage{mathrsfs}

\newtheorem{proposition}{Proposition}[section]

\newtheorem{theorem}[proposition]{Theorem}

\theoremstyle{definition}

\newtheorem{definition}[proposition]{Definition}
\newtheorem{remark}[proposition]{Remark}

\newcommand{\rng}[1]{\langle #1 \rangle}
\newcommand{\C}{\mathscr{C}}
\renewcommand{\P}{\mathscr{P}} 
\renewcommand{\L}{\mathscr{L}} 
\newcommand{\B}{{\mathscr B}}  
\newcommand{\E}{{\mathscr E}}  
\newcommand{\F}{{\mathbb F}}

\newcommand{\GF}{{\mathrm{GF}}}
\newcommand{\GL}{{\mathrm{GL}}}
\newcommand{\PG}{{\mathrm{PG}}}
\newcommand{\Z}{\mathbb{Z}}
\newcommand{\mm}{\mathcal{M}}

\DeclareMathOperator{\gen}{gen}
\DeclareMathOperator{\rank}{rank}

\DeclareMathOperator{\Ker}{Ker}
\DeclareMathOperator{\Img}{Im}
\DeclareMathOperator{\Id}{Id}

\begin{document}

\title[Unital compressed commuting graph of $3 \times 3$ matrices]{Unital compressed commuting graph of $3 \times 3$ matrices over a finite prime field}

\author[I.-V. Boroja]{Ivan-Vanja Boroja \orcidlink{0009-0009-2036-3232}}
\address{University of Banja Luka, Faculty of Natural sciences and Mathematics, Bosnia and Herzegovina}
\email{ivan-vanja.boroja@pmf.unibl.org}

\author[D. Kokol Bukov\v sek]{Damjana Kokol Bukov\v{s}ek \orcidlink{0000-0002-0098-6784}}
\address{University of Ljubljana, School of Economics and Business, and Institute of Mathematics, Physics and Mechanics, Ljubljana, Slovenia}
\email{damjana.kokol.bukovsek@ef.uni-lj.si}

\author[N. Stopar]{Nik Stopar \orcidlink{0000-0002-0004-4957}}
\address{University of Ljubljana, Faculty of Civil and Geodetic Engineering, University of Ljubljana, Faculty of Mathematics and Physics, and Institute of Mathematics, Physics and Mechanics, Ljubljana, Slovenia}
\email{nik.stopar@fgg.uni-lj.si}

\begin{abstract}

In this paper we completely describe the unital compressed commuting graph of the ring $\mm_3(\GF(p))$ of $3 \times 3$ matrices over the finite prime field $\GF(p)$.
To achieve this we combine methods from linear algebra, field theory, projective geometry and combinatorics.
We first partition the set of vertices into types based on the Jordan form and describe the neighborhood of each vertex.
The key part of the graph, i.e., the subgraph that corresponds to non-scalar derogatory matrices, is then determined using a bijective correspondence between its vertices and point-line pairs in the projective plane over $\GF(p)$. At the end we explain how the remaining vertices are attached to the key part. We also give an algorithm to construct the whole graph.
As a consequence, we describe the usual commuting graph $\Gamma(\mm_3(\GF(p)))$, whose structure was an open problem for several years.      
\end{abstract}


\maketitle

\section{Introduction}\label{sec:Intro}

One of the most important notions in algebra is the notion of commutativity. In recent years it has been studied extensively especially in the context of matrices over fields. In the literature there are various approaches to investigating matrix commutativity, such as studying centralizers \cite{DoGuKuOb13,DoGuKuOb14}, commutators \cite{Go10,Br20}, commutativity preservers \cite{OmRaSe01,Se05,ChGoTo25}, etc.

One of the recent approaches is to visualize the relation of commutativity using a graph, the so called \emph{commuting graph} of a ring, that was introduced in \cite{AkGhHaMo04}.
Given a ring $R,$ the commuting graph $\Gamma(R)$ is a simple graph whose vertices are non-central elements of the ring $R$ and two different elements $a$ and $b$ from the ring are connected by an edge if and only if $ab = ba.$ Over the past two decades the commuting graph of a ring was investigated by many researchers who studied connectedness \cite{AkBiMo, DoD19},  diameter and girth \cite{AkMoRaRa06, DoKoKu18, Sh16}, clique number \cite{AkGhHaMo04}, etc. Some authors also investigate the complement of this graph \cite{ErKhNa15}.

One of the central problems in the theory is the isomorphism problem, which asks whether the graph $\Gamma(R)$ uniquely determines the ring $R$, see \cite{AkBiMo}.
For the ring of $n \times n$ matrices over a field this question has a positive answer when $n=2$ and $n=3$ as shown in \cite{Mo10} and \cite{DoMa24}. Also, when $n=2^k3^l$ with $k\geq 1$ a positive answer is given in \cite{DoH19}. For other values of $n$ this is still an open problem.

In \cite{BoDoKoSt23} the compression of the graph $\Gamma(R)$ was studied and the (unital) compressed commuting graph of a ring was introduced.
The unital compressed commuting graph $\Lambda^1(R)$ of a unital ring $R$ is a graph whose vertices are equivalence classes of elements of $R$, with respect to the equivalence relation defined by $a \sim b$ if and only if elements $a$ and $b$ generate the same unital subring of $R$.
Two vertices are adjacent if their respective representatives commute in $R$ (see Definition~\ref{def:UCCG} for details).
The motivation for the compression is to make the graph smaller and thus more manageable, see \cite{BoKoSt24}. Furthermore, it was shown in \cite{BoDoKoSt23} that $\Lambda^1$ can be extended to a functor from the category of unital rings
to the category of undirected simple graphs with loops.
Hence, the compressed commuting graph of a ring $R$ contains also information about commutativity relation in homomorphic images of $R$.
This additional information may be useful when solving the isomorphism problem for matrix rings.

The (unital) compressed commuting graph of the ring of $2 \times 2$ matrices over a finite field was described in \cite{BoDoKoSt23}. In this paper we study the unital compressed commuting graph of the ring $\mm_3(\GF(p))$ of $3 \times 3$ matrices over the finite prime field $\GF(p)$.
In order to describe the vertex set of $\Lambda^1(\mm_3(\GF(p)))$ we show that matrices that are compressed into a single vertex always have the same Jordan structure.
This allows us to describe the vertices case by case and investigate the neighborhood of each vertex.
The key part of the graph $\Lambda^1(\mm_3(\GF(p)))$ is a subgraph that corresponds to non-scalar derogatory matrices. We describe this subgraph by establishing a bijective correspondence between the set of its vertices and the set of point-line pairs in the projective plain over $\GF(p).$ With the help of this correspondence, we then describe the edges between the vertices of this subgraph using the geometry of point-line pairs.  
The rest of the graph is obtained by consecutively attaching vertices corresponding to non-derogatory matrices.
This enables us to give an explicit algorithm for the construction of the entire unital compressed commuting graph of the ring $\mm_3(\GF(p))$. As an application, we are also able to construct the ordinary commuting graph $\Gamma(\mm_3(\GF(p)))$, by blowing-up the vertices of $\Lambda^1(\mm_3(\GF(p)))$ into cliques and removing the center of the ring and all the loops. This algorithmically describes the structure of $\Gamma(\mm_3(\GF(p)))$, which was an open problem for several years.

The paper is organized as follows.
We start by recalling the definitions and some basic properties of the unital compressed commuting graph in Section~\ref{sec:Preliminaries} and recall some results on matrices.
In Section~\ref{sec:Vertices} we describe the vertex set of the unital compressed commuting graph of $\mm_3(\GF(p))$.
The neighborhoods of vertices are considered in Section~\ref{sec:Neighborhoods}.
In Section~\ref{sec:B-E_graph} we describe the key part of the graph, and in Section~\ref{sec:Description} we give the algorithm for the construction of $\Lambda^1(\mm_3(\GF(p)))$.
In final Section~\ref{sec:CG} we apply our results and describe the graph $\Gamma(\mm_3(\GF(p)))$.

\section{Preliminaries}\label{sec:Preliminaries}

We will assume throughout that $R$ is a unital ring with identity element $1$. We first recall the definition of the commuting graph $\Gamma(R)$ and the unital compressed commuting graph $\Lambda^1(R)$.

\begin{definition}[\cite{AkGhHaMo04}]\label{def:CG}
A \emph{commuting graph} of a unital ring $R$ is an undirected graph $\Gamma(R)$ whose vertex set is the set of all non-central elements of $R$ and there is an edge between two different elements $a$ and $b$ if and only if $ab=ba$.
\end{definition}

A unital subring of $R$ generated by an element $a \in R$ is given by
\begin{equation*}
    \rng{a}_1=\{q(a) ~|~ q \in \Z[x]\},
\end{equation*}
where $\Z[x]$ denotes the ring of polynomials with integer coefficients and in the evaluation of $q(a)$ the constant term is multiplied by the identity element $1 \in R$.
We introduce an equivalence relation $\sim$ on $R$ defined by $a \sim b$ if and only if $\rng{a}_1=\rng{b}_1$, and denote the equivalence class of an element $a \in R$ with respect to relation $\sim$ by $[a]_1$.
By definition $[a]_1$ consists of all single generators of the ring $\rng{a}_1$.

\begin{definition}[\cite{BoDoKoSt23}]\label{def:UCCG}
A \emph{unital compressed commuting graph} of a unital ring $R$ is an undirected graph $\Lambda^1(R)$ whose vertex set is the set of all equivalence classes of elements of $R$ with respect to relation $\sim$ and there is an edge between $[a]_1$ and $[b]_1$ if and only if $ab=ba$.
\end{definition}
 
The central elements of $R$ are not excluded from the graph $\Lambda^1(R)$ and loops are allowed, in fact, every vertex of $\Lambda^1(R)$ has a single loop on it.
In \cite{BoDoKoSt23} a non-unital version of the compressed commuting graph was also introduced, however, here we will only be interested in the unital version. 

The mapping $\Lambda^1$ can be extended to a functor $\Lambda^1$ from the category $\mathbf{Ring^1}$ of unital rings and unital ring homomorphisms to the category $\mathbf{Graph}$ of undirected simple graphs that allow loops and graph homomorphisms. 
In particular, for a unital ring homomorphism $f \colon R \to S$, where $R$ and $S$ are unital rings, the graph homomorphism $\Lambda^1(f) \colon \Lambda^1(R) \to \Lambda^1(S)$ is defined by $\Lambda^1(f)([r]_1)=[f(r)]_1$, see \cite{BoDoKoSt23}.
For some further properties of the unital compressed commuting graph we refer the reader to \cite{BoDoKoSt23,BoKoSt24}.

Note that each vertex of $\Lambda^1(R)$ corresponds to a subring of $R$ generated by one element. This means that we can equivalently define the unital compressed commuting graph of $R$ as an undirected graph whose vertex set is the set
\begin{equation*}
    V(\Lambda^1(R))=\{\rng{a}_1 \mid a \in R\}
\end{equation*}
of all unital subrings of $R$ generated by one element, and vertices $\rng{a}_1$ and $\rng{b}_1$ are connected by an edge if and only if $ab=ba$.
In the rest of the paper we will adopt this view and consider vertices of $\Lambda^1(R)$ as unital subrings of $R$ generated by one element. Note that the size of the equivalence class $[a]_1$ equals the number of different single generators of the vertex $\rng{a}_1$.

In this paper we are going to be interested in the ring of $n \times n$ matrices over a field $\F$, which we denote by $\mm_n(\F)$.
The \emph{centralizer} of a matrix $A\in \mm_n(\F)$ is given by
\[ \C(A) = \{ X\in \mm_n(\F) \mid AX = XA  \}.\]
Recall that a matrix is called \emph{non-derogatory} if its minimal polynomial equals its characteristic polynomial up to a sign. A matrix $A$ is non-derogatory if and only if $\C(A)=\F[A]$, where $\F[A]$ is the subalgebra of $\mm_n(\F)$ generated by $A$. In this paper we will be interested in the case when $\F=\GF(p)$, the finite field with $p$ elements, where $p$ is a prime. In this case $\F[A]=\rng{A}_1$. It follows that $\C(A)=\rng{A}_1$ for any non-derogatory matrix $A \in \mm_n(\GF(p))$.

We denote by $\GL_n(\F)$ the group of all invertible matrices in $\mm_n(\F)$ and by
$$\mathcal{O}(A)=\{ SAS^{-1} \mid S\in \GL_n(\F)\}$$
the orbit of a matrix $A \in \mm_n(\F)$ under similarity.
We will need the following well-known formula for the size of the orbit.

\begin{proposition} \label{prop:NumberofmatricesSIMILARtoGIVENmatrix} Let $A\in \mm_n(\F)$, where $\F$ is a finite field. Then 
\[|\mathcal{O}(A)|=\frac{\big|\GL_n(\F)\big|}{\big| \C(A) \cap \GL_n(\F) \big|}.\]
\end{proposition}

Recall also that for a finite field $\F=GF(q)$, with $q=p^m$ elements the number of invertible matrices in $\mm_n(\F)$ is equal to 
\[ |\GL_n(\F)| = (q^n-1)(q^n-q)\cdots(q^n-q^{n-1}).\]

In \cite[Theorem~21]{BoDoKoSt23} the unital compressed commuting graph of the ring $\mm_2(\F)$ was described for any finite field $\F$. In the case when $\F=\GF(p)$ the result reads as follows.

\begin{proposition}[\cite{BoDoKoSt23}] \label{thm:UCCGM2x2} Let $p$ be a prime number. Then the unital compressed commuting graph of the ring $\mm_2(\GF(p))$ is a star graph with $p^2+p+1$ leaves and all the loops.
\end{proposition}

\section{Vertex set of \texorpdfstring{$\Lambda^1(\mm_3(\GF(p)))$}{Λ¹(ℳ₃(\GF(p)))}}\label{sec:Vertices}

We assume henceforth that $p$ is a prime. We will be working with the ring of $3\times 3$ matrices over the field $\GF(p)$.
Our goal is to describe the structure of the unital compressed commuting graph of $\mm_3(\GF(p))$ and we start by describing the vertices of $\Lambda^1(\mm_3(\GF(p)))$.

Let $A \in \mm_3(\GF(p))$ be any matrix. We consider several cases depending on how the characteristic polynomial $p_A$ and minimal polynomial $m_A$ of $A$ split over the field $\GF(p)$, namely:
\begin{enumerate}
    \item[$\mathrm{(A)}$] $p_A(x) = (x-\lambda)^3$ and $m_A(x) = (x-\lambda)$, where $\lambda \in \GF(p)$,
    \item[$\mathrm{(B)}$] $p_A(x) = (x-\lambda)(x-\mu)^2$ and $m_A(x) = (x-\lambda)(x-\mu)$, where $\lambda, \mu \in \GF(p)$ and $\lambda \ne \mu$,
    \item[$\mathrm{(C)}$] $p_A(x) = m_A(x) = (x-\lambda)(x-\mu)(x-\nu)$, where $\lambda, \mu, \nu \in \GF(p)$ are all different, 
    \item[$\mathrm{(D)}$] $p_A(x) = m_A(x) = (x-\lambda)^3$, where $\lambda\in \GF(p)$, 
    \item[$\mathrm{(E)}$] $p_A(x) = (x-\lambda)^3$ and $m_A(x) = (x-\lambda)^2$, where $\lambda\in \GF(p)$,
    \item[$\mathrm{(F)}$] $p_A(x) = m_A(x) = (x-\lambda)^2(x-\mu)$, where $\lambda, \mu\in \GF(p)$ and $\lambda \ne \mu$,
    \item[$\mathrm{(G)}$] $p_A(x) = m_A(x)$ is irreducible over $\GF(p)$, 
    \item[$\mathrm{(H)}$] $p_A(x) = m_A(x) = (x-\lambda)p_1(x)$, where $\lambda\in \GF(p)$ and $p_1$ is irreducible over $\GF(p)$. 
\end{enumerate}

\smallskip
For any $\mathrm{X} \in \{\mathrm{A}, \mathrm{B}, \ldots, \mathrm{H}\}$ we will refer to matrices from case $\mathrm{(X)}$ as matrices of \emph{type} $\mathrm{(X)}$. Next proposition shows that all generators of a vertex in the unital compressed commuting graph are matrices of the same type, so we can also talk about a \emph{vertex of type} $\mathrm{(X)}$.

\begin{proposition}\label{prop:MAT2VERTICES3x3} Suppose $A$ and $B$ are two matrices from $\mm_3(\GF(p)).$ If $\rng{A}_1 = \rng{B}_1$ then $A$ and $B$ are of the same type. \end{proposition}
\begin{proof}
    First note that $\deg m_A = \deg m_B.$ We denote this degree by $d$. As $\rng{A}_1 = \rng{B}_1$ we know that there exist polynomials $q$ and $r$ such that $B=q(A)$ and $A=r(B).$ This implies that matrices $A$ and $B$ have the same number of distinct eigenvalues in $\GF(p).$ We will denote this number by $e.$ Note that $0 \leq e \leq d \leq 3.$ Now, we consider the cases based on the values of $d$ and $e.$

    Note that the pair $(d,e) = (1,0)$ is not possible as linear polynomial always has a zero in the field, i.e., if $d=1$ then $e=1.$ Furthermore, the pair $(d,e) = (2,0)$ is not possible because it would mean that the characteristic polynomial has a double zero $\lambda_1.$ Since this is a zero of minimal polynomial it is an element of $\GF(p^2)\setminus \GF(p)$. But then the second zero $\lambda_2$ of the minimal polynomial must be double as well, which is not possible.
    
    If $(d,e) = (1,1)$ then matrices $A$ and $B$ are of type $\mathrm{(A)}$. Similarly, if $(d,e) = (2,1)$ they are of type $\mathrm{(E)}$, if $(d,e) = (2,2)$ they are of type $\mathrm{(B)}$, if $(d,e) = (3,0)$ they are of type $\mathrm{(G)}$, if $(d,e) = (3,2)$ they are of type $\mathrm{(F)}$, if $(d,e) = (3,3)$ they are of type $\mathrm{(C)}$. Finally, if $(d,e) = (3,1)$ we distinguish two subcases. If matrix $A$ has all eigenvalues in $\GF(p)$ then so does $B$ because $B=q(A).$ In this case both matrices are of type $\mathrm{(D)}$. On the other hand, if $A$ has an eigenvalue $\alpha \notin \GF(p)$ then so does $B$, because $\alpha = r(\beta)$ for some eigenvalue $\beta$ of $B$ and clearly $\beta \notin \GF(p).$  So, both matrices are of type $\mathrm{(H)}$.
\end{proof}

The above proposition allows us to define the following.

\begin{definition}
Define $V_{\mathrm{(X)}} \subseteq V(\Lambda^1(\mm_3(\GF(p))))$ as the set of vertices of type $\mathrm{(X)}$.
\end{definition}

The following proposition gives a formula for calculating the number of vertices of a certain type in $\Lambda^1(\mm_3(\GF(p))$ under a technical condition.

\begin{proposition}\label{prop:V_(X)}
    Suppose we have a vertex of type $\mathrm{(X)}$, generated by a matrix $A$ of order $3$. Assume that $\mathcal{O}(A)$ intersects every vertex of type $\mathrm{(X)}$ and let \begin{equation}
        \omega_A = |\rng{A}_1 \cap \mathcal{O}(A)|.
    \end{equation} Then 
    \begin{equation}\label{eq:N_B}
        |V_{\mathrm{(X)}}| = \frac{|\mathcal{O}(A)|}{\omega_A} =\frac{|\GL_3(\GF(p))|}{|\C(A) \cap \GL_3(\GF(p))|\cdot \omega_A} .
    \end{equation} 
\end{proposition}
\begin{proof}
As shown in the proof of Proposition~\ref{prop:NumberofmatricesSIMILARtoGIVENmatrix}, some of the matrices similar to the matrix $A$ lie inside the subring $\rng{A}_1$, but some of them do not. Those which do not will be generators of the isomorphic copies of the subring $\rng{A}_1$. To count the number of vertices of type $\mathrm{(X)}$, we will need to count how many matrices similar to $A$ lie in the subring $\rng{A}_1.$

Let $M \in \mathcal{O}(A)$ be arbitrary. Then $M$ is similar to $A,$ i.e., there exists an invertible matrix $S$ such that $M=SAS^{-1}.$ The conjugation mapping $Y \mapsto SYS^{-1}$ is bijection from set $\rng{A}_1 \cap \mathcal{O}(A)$ to the set $\rng{M}_1 \cap \mathcal{O}(A)$, so all of the sets $\rng{M}_1 \cap \mathcal{O}(A)$, $M\in \mathcal{O}(A)$, are of the same cardinality and this cardinality is equal to $\omega_A.$  

Note that any $W$ from $\rng{M}_1 \cap \mathcal{O}(A)$ is automatically a generator of $\rng{M}_1$. This is because $W$ is similar to $A$ and hence similar to $M$, so three of them have the same degree of minimal polynomial. Since $W$ is from $\rng{M}_1$ this implies that $W$ is a generator of $\rng{M}_1$. Hence, matrices in $\rng{M}_1\cap \mathcal{O}(A)$ are compressed into the same vertex.

It follows that the number of vertices obtained from matrices of type $\mathrm{(X)}$ is equal to 
\begin{equation*}
    |V_{\mathrm{(X)}}| = \frac{|\mathcal{O}(A)|}{\omega_A}.
\end{equation*}   
Equation \eqref{eq:N_B} now follows from Proposition~\ref{prop:NumberofmatricesSIMILARtoGIVENmatrix}.
\end{proof}

For each matrix $A \in \mm_3(\GF(p))$ of a certain type $\mathrm{(X)}$ we compute $\dim\rng{A}_1$, the number of generators of $\rng{A}_1$, and the cardinality of $V_{\mathrm{(X)}}$. Since the dimension and the number of generators of $\rng{A}_1$ does not change under similarity, we may assume that matrix $A$ is in Jordan form.

\medskip

{\bf Case (A)}: Assume that matrix $A$ satisfies $p_A(x) = (x-\lambda)^3$ and $m_A(x) = (x-\lambda)$, where $\lambda \in \GF(p)$. Then $A$ is a scalar matrix, the subring generated by $A$ is the subring of all scalar matrices and every element of this subring is a generator. Hence, $|V_{\mathrm{(A)}}|=1$, $\dim \rng{A}_1 =1$, and $\rng{A}_1$ has $p$ generators.

\medskip

{\bf Case (B)}: Assume that matrix $A$ satisfies $p_A(x) = (x-\lambda)(x-\mu)^2$ and $m_A(x) = (x-\lambda)(x-\mu)$, where $\lambda, \mu \in \GF(p)$ and $\lambda \ne \mu$. It follows that matrix $A$ is diagonalizable, so we assume it is diagonal, i.e., $A=\lambda E_{11}+\mu(E_{22}+E_{33})$.
It follows that $\dim \rng{A}_1 = 2$ and
\begin{equation*}
  \rng{A}_1=\{\alpha E_{11}+\beta(E_{22}+E_{33}) \mid \alpha, \beta \in \GF(p)\}.
\end{equation*}
Note that an element of $\rng{A}_1$ is a generator if and only if $\beta \neq \alpha$. So, the number of generators of $\rng{A}_1$ is $p(p-1).$

To calculate the cardinality of $V_{\mathrm{(B)}}$ we will use Proposition~\ref{prop:V_(X)}, so we need to show that ${\mathcal O}(A)$ intersects every vertex of type $\mathrm{(B)}$. Let $Y$ be an arbitrary matrix of type $\mathrm{(B)}$. This means that there exists an invertible matrix $S$ such that $\widehat{A}=SYS^{-1}$ is a generator of $\rng{A}_1$.
Hence, 
\begin{equation*}
    \rng{Y}_1 = \rng{S^{-1}\widehat{A}S}_1 = S^{-1}\rng{\widehat{A}}_1S = S^{-1}\rng{A}_1S = \rng{S^{-1}AS}_1,
\end{equation*}
so that $S^{-1}AS \in \rng{Y}_1 \cap \mathcal{O}(A).$
Furthermore, the centralizer of the matrix $A$ is
\begin{equation}\label{eq:CentralizerTypeB}
    \C(A)=\Bigg \{ \begin{bmatrix}
        a & 0 & 0 \\
        0 & e & f \\
        0 & h & i
    \end{bmatrix} \mid a,e,f,h,i \in \GF(p) \Bigg \}.
\end{equation}
Hence,
\begin{equation*}
    |\C(A) \cap \GL_3(\GF(p))|= |\GL_1(\GF(p))|\cdot |\GL_2(\GF(p))|=(p-1)(p^2-1)(p^2-p).
\end{equation*}
We now prove that $\omega_A=1.$ Let $M\in \rng{A}_1\cap \mathcal{O}(A)$ be arbitrary. Since $M\in \rng{A}_1$ we infer that $M=\alpha E_{11}+\beta(E_{22}+E_{33})$ for some $\alpha, \beta \in \GF(p).$ Since $M\in \mathcal{O}(A)$, we have $\alpha = \lambda$ and $\beta = \mu$. We conclude that $M=A$, which proves our claim.
By Proposition~\ref{prop:V_(X)} the number of vertices of type $\mathrm{(B)}$ is 
\[ |V_{\mathrm{(B)}}| = \frac{|\mathcal{O}(A)|}{\omega_A}=\frac{|\GL_3(\GF(p))|}{|\C(A) \cap \GL_3(\GF(p))|\cdot \omega_A} =\frac{(p^3-1)(p^3-p)(p^3-p^2)}{(p-1)(p^2-1)(p^2-p)\cdot 1}=(p^2+p+1)p^2.\] 

\medskip 

{\bf Case (C)}: Assume that matrix $A$ satisfies $p_A(x) = m_A(x) = (x-\lambda)(x-\mu)(x-\nu)$, where $\lambda, \mu, \nu \in \GF(p)$ are all different. It follows that matrix $A$ is diagonalizable, so we assume it is diagonal, i.e., $A=\lambda E_{11}+\mu E_{22}+\nu E_{33}$.
Note that in this case $p$ must be greater then or equal to $3$ because if $p=2$ a matrix cannot have three different eigenvalues in $\GF(p).$
It follows that $\dim \rng{A}_1 = 3$ and
\begin{equation*}
  \rng{A}_1=\{\alpha E_{11}+\beta E_{22}+\gamma E_{33} \mid \alpha, \beta, \gamma \in \GF(p)\}.
\end{equation*}
Note that an element of $\rng{A}_1$ is a generator if and only if $\alpha, \beta$ and $\gamma$ are all different. So, the number of generators of $\rng{A}_1$ is $p(p-1)(p-2).$

Note also that the subring $\rng{A}_1$ does not depend on the exact values of $\lambda$, $\mu$ and $\nu.$ This implies that we can use the same argument as in case $\mathrm{(B)}$ and conclude that ${\mathcal O}(A)$ intersects every vertex of type $\mathrm{(C)}$. This means that the assumption of Proposition~\ref{prop:V_(X)} is satisfied. 
Since matrix $A$ is non-derogatory we have $\C(A)=\rng{A}_1$, so that
\begin{equation*}
    |\C(A) \cap \GL_3(\GF(p))|= (p-1)^3.
\end{equation*}
Next, we calculate $\omega_A$.
Let $M\in \rng{A}_1 \cap \mathcal{O}(A)$. Then $M=\alpha E_{11}+\beta E_{22}+\gamma E_{33}$ and $\{\alpha,\beta,\gamma\}=\{\lambda, \mu, \nu\}$. It follows that $(\alpha,\beta,\gamma)$ is a permutation of $(\lambda, \mu, \nu)$, hence
$\omega_A = |\rng{A}_1 \cap \mathcal{O}(A)|=6$.
By Proposition~\ref{prop:V_(X)} we conclude that the number of vertices of type $\mathrm{(C)}$ is equal to
\begin{equation*}
\begin{aligned}
    |V_{\mathrm{(C)}}| &= \frac{|\mathcal{O}(A)|}{\omega_A}= \frac{|\GL_3(\GF(p))|}{|\C(A) \cap \GL_3(\GF(p))|\cdot \omega_A} = \frac{(p^3-1)(p^3-p)(p^3-p^2)}{(p-1)^3\cdot 6}\\ &= \tfrac{1}{6}(p^2+p+1)p^3(p+1).
\end{aligned}
\end{equation*}

\medskip

{\bf Case (D)}: Assume that matrix $A$ satisfies $p_A(x) = m_A(x) = (x-\lambda)^3$, where $\lambda\in \GF(p)$. We may assume that $A$ is in Jordan form, i.e., $A=\lambda I+ E_{12}+E_{23}$.
It follows that $\dim \rng{A}_1 = 3$ and
\begin{equation*}
  \rng{A}_1=\{\alpha I+ \beta(E_{12}+E_{23})+\gamma E_{13} \mid \alpha, \beta, \gamma \in \GF(p)\}.
\end{equation*}
Note that an element of $\rng{A}_1$ is a generator if and only if $\beta \ne 0$. So, the number of generators of $\rng{A}_1$ is $p^2(p-1).$
Once again the subring $\rng{A}_1$ does not depend on the value of $\lambda$, so, as in previous cases, the assumption of Proposition~\ref{prop:V_(X)} is fulfilled. 
Since matrix $A$ is non-derogatory we have $\C(A)=\rng{A}_1$, so that
\begin{equation*}
    |\C(A) \cap \GL_3(\GF(p))|= (p-1)p^2.
\end{equation*}
Next, we calculate $\omega_A$.
Let $M\in \rng{A}_1 \cap \mathcal{O}(A)$. Then $M=\alpha I+ \beta(E_{12}+E_{23})+\gamma E_{13}$. Note that $M$ is similar to $A$, if and only if $\alpha=\lambda$ and $\beta \ne 0$,
therefore,
\[\omega_A = |\rng{A}_1 \cap \mathcal{O}(A)|= (p-1)p.\]
Finally, the number of vertices of type $\mathrm{(D)}$ equals
\begin{equation*}
\begin{aligned}
    |V_{\mathrm{(D)}}| &= \frac{|\mathcal{O}(A)|}{\omega_A}= \frac{|\GL_3(\GF(p))|}{|\C(A) \cap \GL_3(\GF(p))|\cdot \omega_A} = \frac{(p^3-1)(p^3-p)(p^3-p^2)}{(p-1)p^2\cdot (p-1)p}\\ &= (p^3-1)(p+1).
\end{aligned}
\end{equation*}

\medskip

{\bf Case (E)}: Assume that matrix $A$ satisfies $p_A(x) = (x-\lambda)^3$ and $m_A(x) = (x-\lambda)^2$, where $\lambda\in \GF(p)$. We may assume that $A$ is in Jordan form, i.e., $A=\lambda I+ E_{12}$.
It follows that $\dim \rng{A}_1 = 2$ and
\begin{equation*}
  \rng{A}_1=\{\alpha I+ \beta E_{12} \mid \alpha, \beta \in \GF(p)\}.
\end{equation*}
Note that an element of $\rng{A}_1$ is a generator if and only if $\beta \ne 0$. So, the number of generators of $\rng{A}_1$ is $p(p-1).$
The subring $\rng{A}_1$ still does not depend on the specific value of $\lambda$, so the assumption of Proposition~\ref{prop:V_(X)} is satisfied.

Furthermore, direct calculation shows that the centralizer of $A$ equals \begin{equation}\label{eq:CentralizerTypeE}
    \C(A)=\Bigg\{ \begin{bmatrix}
    a & b & c \\ 0 & a & 0 \\ 0 & h & i
\end{bmatrix} \mid a,b,c,h,i \in \GF(p) \Bigg \}.
\end{equation}
A matrix from $\C(A)$ is invertible if and only if $a\neq 0$ and $i\neq 0$, so that
\begin{align*}
    \big|\C(A)\cap \GL_3(\GF(p)) \big| =(p-1)^2p^3.
\end{align*} 
In order to compute $\omega_A$, let $M\in \rng{A}_1 \cap \mathcal{O}(A)$. We have $M=\alpha I+ \beta E_{12}$, and since $M$ is similar to $A,$ we have $\alpha = \lambda$ and $\beta \neq 0.$ 
Therefore, $\omega_A = |\rng{A}_1 \cap \mathcal{O}(A)|=p-1$.
We conclude that the number of vertices of type $\mathrm{(E)}$ is equal to
\begin{equation*}
\begin{aligned}
    |V_{\mathrm{(E)}}| &= \frac{|\mathcal{O}(A)|}{\omega_A}= \frac{|\GL_3(\GF(p))|}{|\C(A) \cap \GL_3(\GF(p))|\cdot \omega_A} = \frac{(p^3-1)(p^3-p)(p^3-p^2)}{(p-1)^2p^3\cdot (p-1)}\\ &= (p^2+p+1)(p+1).
\end{aligned}
\end{equation*}

\medskip

{\bf Case (F)}: Assume that matrix $A$ satisfies $p_A(x) = m_A(x) = (x-\lambda)^2(x-\mu)$, where $\lambda, \mu\in \GF(p)$ and $\lambda \ne \mu$. We may assume that $A$ is in Jordan form, i.e., $A=\lambda(E_{11}+E_{22})+E_{12}+ \mu E_{33}$.
It follows that $\dim \rng{A}_1 = 3$ and
\begin{equation*}
  \rng{A}_1=\{\alpha(E_{11}+E_{22})+ \beta E_{12}+ \gamma E_{33} \mid \alpha, \beta, \gamma \in \GF(p)\}.
\end{equation*}
Note that an element of $\rng{A}_1$ is a generator if and only if $\alpha \ne \gamma$ and $\beta \ne 0$. So, the number of generators of $\rng{A}_1$ is $p(p-1)^2.$

Note also that $\rng{A}_1$ does not depend on the specific values of $\lambda$ and $\mu$, so again the assumption of Proposition~\ref{prop:V_(X)} is satisfied. Since matrix $A$ is non-derogatory, we have $\C(A)=\rng{A}_1$, hence,
$$|\C(A) \cap \GL_3(\GF(p))| = (p-1)^2p.$$

Let $M\in \rng{A}_1 \cap \mathcal{O}(A)$. Then $M=\alpha(E_{11}+E_{22})+ \beta E_{12}+ \gamma E_{33}$ and since $M$ is similar to $A$ we have $\alpha = \lambda$, $\gamma=\mu$ and $\beta \neq 0.$ 
Therefore,
\[\omega_A = |\rng{A}_1 \cap \mathcal{O}(A)|=p-1.\] 
By Proposition~\ref{prop:V_(X)} we conclude that the number of vertices of type $\mathrm{(F)}$ is equal to
\begin{equation*}
\begin{aligned}
    |V_{\mathrm{(F)}}| &= \frac{|\mathcal{O}(A)|}{\omega_A}= \frac{|\GL_3(\GF(p))|}{|\C(A) \cap \GL_3(\GF(p))|\cdot \omega_A} = \frac{(p^3-1)(p^3-p)(p^3-p^2)}{(p-1)^2p\cdot (p-1)}\\ &= (p^2+p+1)p^2(p+1).
\end{aligned}
\end{equation*}

\medskip

{\bf Case (G)}: Assume that matrix $A$ is such that $p_A(x) = m_A(x)$ is irreducible over $\GF(p)$. It follows that $\dim \rng{A}_1 = 3$. Since $m_A$ is irreducible, $\rng{A}_1 \cong \GF(p^3)$ (see for example \cite{He99}), so the number of generators of $\rng{A}_1$ is $p^3-p$.
Note that matrix $A$ is non-derogatory, so $\C(A)=\rng{A}_1$ and
$$|\C(A) \cap \GL_3(\GF(p))| = p^3-1.$$

We continue in a slightly different way than in the previous cases. From Proposition~\ref{prop:NumberofmatricesSIMILARtoGIVENmatrix} we know that \[\displaystyle |\mathcal{O}(A)|=\frac{\big|\GL_3(\GF(p))\big|}{\big| \C(A) \cap \GL_3(\GF(p)) \big|} = \frac{(p^3-1)(p^3-p)(p^3-p^2)}{(p^3-1)} = (p^3-p) (p^3-p^2).\]
Since $A$ is similar to the companion matrix of polynomial $p_A$, the orbit of every matrix of type $\mathrm{(G)}$ contains a companion matrix of an irreducible polynomial. In fact, it contains exactly one companion matrix,
because if $C_1$ and $C_2$ are companion matrices of two irreducible polynomials of degree $3$ contained in the same orbit, then $C_1$ is similar to $C_2$ which implies that $p_{C_1} = p_{C_2}$ and hence $C_1=C_2.$
So, the number of orbits in this case is the same as the number of monic irreducible polynomials of degree 3, which is equal to $\frac{p^3-p}{3}$ by \cite[Theorem 3.25]{LiNi94}.
This means that the number of all matrices of type $\mathrm{(G)}$ is equal to \begin{equation*}
    \frac{p^3-p}{3}\cdot |\mathcal{O}(A)| = \frac{p^3-p}{3} \cdot (p^3-p) (p^3-p^2) = \frac{(p^3-p)(p^3-p) (p^3-p^2)}{3}.
\end{equation*}
If we divide this number by the number of generators of a vertex of type $\mathrm{(G)}$ we get the number of vertices of type $\mathrm{(G)}$
\begin{equation*}
\begin{aligned}
    |V_{\mathrm{(G)}}| = \frac{(p^3-p)(p^3-p) (p^3-p^2)}{3(p^3-p)}= \tfrac13(p^3-p)(p^3-p^2).
\end{aligned}
\end{equation*}

\medskip

{\bf Case (H)}: Assume that matrix $A$ satisfies $p_A(x) = m_A(x) = (x-\lambda)p_1(x)$, where $\lambda\in \GF(p)$ and $p_1$ is irreducible over $\GF(p)$. It follows that $\dim \rng{A}_1 = 3$. 
Matrix $A$ is similar to block diagonal matrix
\begin{equation}\label{eq:CompanionMatrixTypeH}
\begin{bmatrix} 
  \lambda & 0  \\
  0 & C
\end{bmatrix},
\end{equation}
where $C$ is the companion matrix of $p_1$. Since $p_1$ is irreducible, $\rng{C}_1 \cong \GF(p^2)$, so
$$\rng{A}_1 \cong \GF(p) \oplus \GF(p^2).$$
Consequently, the number of generators of $\rng{A}_1$ is $p(p^2-p)$. Note that matrix $A$ is non-derogatory, so $\C(A)=\rng{A}_1$. Furthermore, $B =(b_1,b_2) \in \GF(p) \oplus \GF(p^2)$ is invertible if and only if $b_1 \ne 0$ and $b_2 \ne 0$, so
$$|\C(A) \cap \GL_3(\GF(p))| = (p-1)(p^2-1).$$
From Proposition~\ref{prop:NumberofmatricesSIMILARtoGIVENmatrix} we conclude that \[\displaystyle |\mathcal{O}(A)|=\frac{\big|\GL_3(\GF(p))\big|}{\big| \C(A) \cap \GL_3(\GF(p)) \big|} = \frac{(p^3-1)(p^3-p)(p^3-p^2)}{(p^2-1)(p-1)} = (p^3-1)p^3.\] 

By similar arguments as in case $\mathrm{(G)}$, the orbit of every matrix of type $\mathrm{(H)}$ contains exactly one matrix of the form \eqref{eq:CompanionMatrixTypeH}. So, the number of orbits in this case is equal to the number of polynomials of the form $(x-\lambda)p_1(x)$, where $\lambda\in \GF(p)$ and $p_1$ is irreducible over $\GF(p)$. This number is equal to $p\cdot \tfrac{p^2-p}{2}$ by \cite[Theorem 3.25]{LiNi94}.
This implies that the number of matrices of type $\mathrm{(H)}$ equals
\begin{equation*}
    p\cdot \frac{p^2-p}{2} \cdot |\mathcal{O}(A)| = p\cdot \frac{p^2-p}{2} \cdot (p^3-1)p^3 = \frac{p^5(p-1)(p^3-1)}{2}.
\end{equation*}
Dividing by the number of generators of $\rng{A}_1$, we obtain
\begin{equation*}
    |V_{\mathrm{(H)}}| = \frac{p^5(p-1)(p^3-1)}{2p^2(p-1)} = \tfrac12 p^3(p^3-1).
\end{equation*}

As we have seen above, case by case, if two vertices are of the same type, then they have the same number of single generators. This allows us to introduce the following notation.

\begin{definition}
We denote by $\gen_{\mathrm{(X)}}$ the number of single generators of a vertex of type $\mathrm{(X)}$.
\end{definition}

We now summarize the results from cases $\mathrm{(A)}$--$\mathrm{(H)}$ in next proposition.

\begin{proposition}
    Let $p$ be a prime and $\mm_3(\GF(p))$ the ring of $3 \times 3$ matrices over $\GF(p)$. For every $\mathrm{X} \in \{\mathrm{A}, \mathrm{B}, \ldots, \mathrm{H}\}$, Table~\ref{tab:RedTable} gives the number of vertices of type $(\mathrm{X})$ in the unital compressed commuting graph $\Lambda^1(\mm_3(\GF(p)))$, the number of single generators of each vertex of type $(\mathrm{X})$, and its dimension over $\GF(p)$.
\end{proposition}

Note that the scalar product of the columns $|V_{\mathrm{(X)}}|$ and $\gen_{\mathrm{(X)}}$ of Table~\ref{tab:RedTable} equals $p^9$, the cardinality of $\mm_3(\GF(p)).$

\begin{table}[h!]
    \caption{Vertices of $\Lambda^1(\mm_3(\GF(p))).$}
    \centering
    \begin{tabular}{c|ccc}
     $(\mathrm{X})$ & $|V_{\mathrm{(X)}}|$ & $\gen_{\mathrm{(X)}}$ &  $\dim_{(\mathrm{X})}$ \\ \hline \hline
      $\mathrm{(A)}$ & $1$  & $p$ & $1$ \\
      $\mathrm{(B)}$ & $(p^2+p+1)p^2$ & $p(p-1)$ & $2$ \\ 
      $\mathrm{(C)}$ & $\frac{1}{6}(p^2+p+1)p^3(p+1)$ & $p(p-1)(p-2)$ & $3$ \\ 
      $\mathrm{(D)}$ & $(p^3-1)(p+1)$ & $p^2(p-1)$ & $3$ \\
      $\mathrm{(E)}$ & $(p^2+p+1)(p+1)$ & $p(p-1)$ & $2$ \\
      $\mathrm{(F)}$ & $(p^2+p+1)p^2(p+1)$ & $p(p-1)^2$ & $3$ \\
      $\mathrm{(G)}$ & $\frac{1}{3}(p^3-p)(p^3-p^2)$ & $p^3-p$ & $3$ \\
      $\mathrm{(H)}$ & $\frac{1}{2}(p^3-1)p^3$ & $p^2(p-1)$ & $3$ \\
    \end{tabular} \vspace{4mm}
    \label{tab:RedTable}
\end{table}  

\begin{remark}
    In cases $\mathrm{(B)},\mathrm{(C)},\mathrm{(D)},\mathrm{(E)},$ and $\mathrm{(F)}$, we have shown that the orbit $\mathcal{O}(A)$ of a generator of a vertex intersects every vertex of the same type.
    This is obviously also true in case $\mathrm{(A)}$ and it can be shown that it holds also in cases $\mathrm{(G)}$ and $\mathrm{(H)}$.
    So we can omit the assumption about the orbit in Proposition~\ref{prop:V_(X)}.
\end{remark}

\section{Neighborhoods of vertices of \texorpdfstring{$\Lambda^1(\mm_3(\GF(p)))$}{Λ¹(ℳ₃(\GF(p)))}}\label{sec:Neighborhoods}
\rm

In this section we describe the neighborhood of each vertex of $\Lambda^1(\mm_3(\GF(p))).$ For a vertex $v$ of a certain type we will calculate the number of vertices of each type that are connected to $v$. Note that the neighborhood of vertex $v=\rng{A}_1$ is a unital compressed commuting graph of the centralizer $\C(A)$. For a vertex $v$ of type $\mathrm{(Y)}$ we will denote by $N(\mathrm{X},\mathrm{Y})$ the number of neighbors of $v$ of type $\mathrm{(X)}$.
We consider the cases with $\dim \rng{A}_1$ equal to $1$ and $3$ first.

\medskip{\bf Case (A)}: A matrix of type $\mathrm{(A)}$ is a scalar matrix so its centralizer is the whole $\mm_3(\GF(p))$. It follows that $N(\mathrm{X},\mathrm{A})=|V_{\mathrm{(X)}}|$ for any type $\mathrm{(X)}$.

{\bf Case (C)}: 
A matrix $A$ of type $\mathrm{(C)}$ is diagonalizable, so we may assume it is diagonal. Then, the centralizer of $A$ equals
$$\C(A)=\rng{A}_1=\{\alpha E_{11}+\beta E_{22}+\gamma E_{33} \mid \alpha, \beta, \gamma \in \GF(p)\}.$$
Let $B=\alpha E_{11}+\beta E_{22}+\gamma E_{33} \in \C(A)$. If $\alpha=\beta=\gamma$ then $B$ is of type $\mathrm{(A)}$. If $\beta = \gamma \neq \alpha$, $\alpha=\beta\neq \gamma$ or $\alpha=\gamma\neq \beta$ then $B$ is of type $\mathrm{(B)}$ and the three cases give three different vertices. If $\alpha,\beta$ and $\gamma$ are all different then $\rng{B}_1=\rng{A}_1$.
It follows that
$$N(\mathrm{A},\mathrm{C})=1, \qquad N(\mathrm{B},\mathrm{C})=3, \qquad N(\mathrm{C},\mathrm{C})=1,$$
and
$$N(\mathrm{X},\mathrm{C})=0 \qquad \text{for all $\mathrm{X} \in \{\mathrm{D}, \mathrm{E}, \mathrm{F}, \mathrm{G}, \mathrm{H}\}$}.$$

{\bf Case (D)}: 
We may assume that a matrix $A$ of type $\mathrm{(D)}$ is in Jordan form. Recall that the centralizer of $A$ is
$$\C(A)=\rng{A}_1=\{\alpha I+ \beta(E_{12}+E_{23})+\gamma E_{13} \mid \alpha, \beta, \gamma \in \GF(p)\}.$$
Let $B=\alpha I+ \beta(E_{12}+E_{23})+\gamma E_{13} \in \C(A)$. If $\beta=\gamma=0$ then $B$ is of type $\mathrm{(A)}$. If $\beta=0$ and $\gamma \ne 0$ then $B$ is of type $\mathrm{(E)}$ and all such matrices generate the same subring. If $\beta \ne 0$ then $\rng{B}_1=\rng{A}_1$.
It follows that
$$N(\mathrm{A},\mathrm{D})=1, \qquad N(\mathrm{D},\mathrm{D})=1, \qquad N(\mathrm{E},\mathrm{D})=1,$$
and
$$N(\mathrm{X},\mathrm{D})=0 \qquad \text{for all $\mathrm{X} \in \{\mathrm{B}, \mathrm{C}, \mathrm{F}, \mathrm{G}, \mathrm{H}\}$}.$$

{\bf Case (F)}: We may assume that a matrix $A$ of type $\mathrm{(E)}$ is in Jordan form. Recall that the centralizer of $A$ is
$$\C(A)=\rng{A}_1=\{\alpha(E_{11}+E_{22})+ \beta E_{12}+ \gamma E_{33} \mid \alpha, \beta, \gamma \in \GF(p)\}.$$
Let $B=\alpha(E_{11}+E_{22})+ \beta E_{12}+ \gamma E_{33} \in \C(A)$. If $\alpha=\beta$ and $\gamma=0$ then $B$ is of type $\mathrm{(A)}$. If $\alpha=\beta$ and $\gamma \ne 0$ then matrix $B$ is of type $\mathrm{(E)}$ and all such matrices generate the same subring. If $\alpha \ne \beta$ and $\gamma=0$ then $B$ is of type $\mathrm{(B)}$ and all such matrices generate the same subring. Finally, if $\alpha \ne \beta$ and $\gamma \ne 0$ then $\rng{B}_1=\rng{A}_1$.
It follows that
$$N(\mathrm{A},\mathrm{F})=1, \qquad N(\mathrm{B},\mathrm{F})=1, \qquad N(\mathrm{E},\mathrm{F})=1, \qquad N(\mathrm{F},\mathrm{F})=1$$
and
$$N(\mathrm{X},\mathrm{F})=0 \qquad \text{for all $\mathrm{X} \in \{\mathrm{C}, \mathrm{D}, \mathrm{G}, \mathrm{H}\}$}.$$

{\bf Case (G)}: If $A$ is of type $\mathrm{(G)}$ then $\C(A)=\rng{A}_1 \cong \GF(p^3)$. This field has only two subrings generated by one element, namely $\GF(p)$, which corresponds to a vertex of type $\mathrm{(A)}$, and $\GF(p^3)$, which corresponds to $\rng{A}_1$. Hence,
$$N(\mathrm{A},\mathrm{G})=1, \qquad N(\mathrm{G},\mathrm{G})=1,$$
and
$$N(\mathrm{X},\mathrm{G})=0 \qquad \text{for all $\mathrm{X} \in \{\mathrm{B}, \mathrm{C}, \mathrm{D}, \mathrm{E}, \mathrm{F}, \mathrm{H}\}$}.$$

{\bf Case (H)}: Recall that a matrix $A$ of type $H$ can be assumed to be in a block diagonal form
$$A=\begin{bmatrix} 
  \lambda & 0  \\
  0 & C
\end{bmatrix},$$
where $\lambda \in \GF(p)$ and $\rng{C}_1 \cong \GF(p^2)$. Hence, the centralizer of $A$ equals
$$\C(A)=\left\{
\begin{bmatrix} 
  \alpha & 0  \\
  0 & \beta I+\gamma C
\end{bmatrix} \mid \alpha, \beta, \gamma \in \GF(p)
\right\}.$$
Let $B=\begin{bmatrix} 
  \alpha & 0  \\
  0 & \beta I+\gamma C
\end{bmatrix} \in \C(A)$. If $\gamma=0$ and $\alpha=\beta$ then $B$ is of type $\mathrm{(A)}$. If $\gamma=0$ and $\alpha \ne \beta$  then $B$ is of type $\mathrm{(B)}$ and all such matrices generate the same subring. Finally, if $\gamma \ne 0$ then $\rng{B}_1=\rng{A}_1$.
It follows that
$$N(\mathrm{A},\mathrm{H})=1, \qquad N(\mathrm{B},\mathrm{H})=1, \qquad N(\mathrm{H},\mathrm{H})=1,$$
and
$$N(\mathrm{X},\mathrm{H})=0 \qquad \text{for all $\mathrm{X} \in \{\mathrm{C}, \mathrm{D}, \mathrm{E}, \mathrm{F}, \mathrm{G}\}$}.$$

{\bf Case (E)}: We may assume that the matrix $A$ of type $\mathrm{(E)}$ is in its Jordan form. Recall that the centralizer of $A$ is
\begin{equation*}
\C(A)=\Bigg\{ \begin{bmatrix}
    a & b & c \\ 0 & a & 0 \\ 0 & h & i
\end{bmatrix} \mid a,b,c,h,i \in \GF(p) \Bigg \}.
\end{equation*}
Let $B \in \C(A)$.
First, since $|V_{\mathrm{(A)}}|=1$ and matrix $B$ cannot be of type $\mathrm{(C)}$, $\mathrm{(G)}$ or $\mathrm{(H)}$, we have
$$N(\mathrm{A},\mathrm{E})=1, \qquad N(\mathrm{C},\mathrm{E})=0, \qquad N(\mathrm{G},\mathrm{E})=0 \qquad \text{and} \qquad N(\mathrm{H},\mathrm{E})=0.$$
The number of all edges between $V_{\mathrm{(E)}}$ and $V_{\mathrm{(X)}}$ can be calculated in two different ways, namely
$$N(\mathrm{X},\mathrm{E}) \cdot |V_{\mathrm{(E)}}|=N(\mathrm{E},\mathrm{X}) \cdot |V_{\mathrm{(X)}}|.$$
It follows that
\begin{align*}
N(\mathrm{D},\mathrm{E}) &=\frac{N(\mathrm{E},\mathrm{D}) \cdot |V_{\mathrm{(D)}}|}{|V_{\mathrm{(E)}}|}=\frac{1 \cdot (p^3-1)(p+1)}{(p^2+p+1)(p+1)}=p-1,\\
N(\mathrm{F},\mathrm{E}) &=\frac{N(\mathrm{E},\mathrm{F}) \cdot |V_{\mathrm{(F)}}|}{|V_{\mathrm{(E)}}|}=\frac{1 \cdot (p^2+p+1)p^2(p+1)}{(p^2+p+1)(p+1)}=p^2.
\end{align*}
Recall that $\gen_{\mathrm{(X)}}$ denotes the number of generators of a vertex of type $\mathrm{(X)}$.
We now consider two cases. If $a=i$, then $B$ is of type $\mathrm{(A)}$, $\mathrm{(D)}$ or $\mathrm{\mathrm{(E)}}$. There are $p^4$ such matrices in $\C(A)$, thus
$$N(\mathrm{A},\mathrm{E}) \cdot \gen_{\mathrm{(A)}} + N(\mathrm{D},\mathrm{E}) \cdot \gen_{\mathrm{(D)}} + N(\mathrm{E},\mathrm{E}) \cdot \gen_{\mathrm{(E)}} = p^4.$$
It follows that
\begin{align*}
    N(\mathrm{E},\mathrm{E}) &=\frac{p^4- N(\mathrm{A},\mathrm{E}) \cdot \gen_{\mathrm{(A)}} - N(\mathrm{D},\mathrm{E}) \cdot \gen_{\mathrm{(D)}}}{\gen_{\mathrm{(E)}}}\\
    &=\frac{p^4- 1 \cdot p - (p-1) \cdot p^2(p-1)}{p(p-1)} = 2p+1.
\end{align*}
If $a\ne i$, then $B$ is of type $\mathrm{(B)}$ or $\mathrm{(F)}$. There are $p^4(p-1)$ such matrices in $\C(A)$, thus
$$N(\mathrm{B},\mathrm{E}) \cdot \gen_{\mathrm{(B)}} + N(\mathrm{F},\mathrm{E}) \cdot \gen_{\mathrm{(D)}} = p^4(p-1).$$
It follows that
\begin{align*}
    N(\mathrm{B},\mathrm{E}) &=\frac{p^4(p-1)- N(\mathrm{F},\mathrm{E}) \cdot \gen_{\mathrm{(F)}}}{\gen_{\mathrm{(B)}}}=\frac{p^4(p-1) - p^2 \cdot p(p-1)^2}{p(p-1)} = p^2.
\end{align*}

{\bf Case (B)}: Matrix $A$ of type $\mathrm{(B)}$ is diagonalizable, so we may assume it is in diagonal form. Recall that the centralizer of $A$ is
\begin{equation*}
    \C(A)=\left \{ \begin{bmatrix}
        \alpha & 0\\
         0 & C
    \end{bmatrix} \mid \alpha \in \GF(p), \ C \in \mm_2(\GF(p)) \right \}.
\end{equation*}
Let $B=\begin{bmatrix}
        \alpha & 0\\
         0 & C
    \end{bmatrix} \in \C(A)$.
First, since $|V_{\mathrm{(A)}}|=1$ and matrix $B$ cannot be of type $\mathrm{(D)}$ or $\mathrm{(G)}$, we have
$$N(\mathrm{A},\mathrm{B})=1, \qquad N(\mathrm{D},\mathrm{B})=0 \qquad \text{and} \qquad N(\mathrm{G},\mathrm{B})=0.$$
Similarly as in Case $\mathrm{(E)}$, we have
$$N(\mathrm{X},\mathrm{B}) \cdot |V_{\mathrm{(B)}}|=N(\mathrm{B},\mathrm{X}) \cdot |V_{\mathrm{(X)}}|.$$
It follows that
\begin{align*}
N(\mathrm{C},\mathrm{B}) &=\frac{N(\mathrm{B},\mathrm{C}) \cdot |V_{\mathrm{(C)}}|}{|V_{\mathrm{(B)}}|}=\frac{3 \cdot \frac{1}{6}(p^2+p+1)p^3(p+1)}{(p^2+p+1)p^2}=\tfrac12 p(p+1),\\
N(\mathrm{E},\mathrm{B}) &=\frac{N(\mathrm{B},\mathrm{E}) \cdot |V_{\mathrm{(E)}}|}{|V_{\mathrm{(B)}}|}=\frac{p^2 \cdot (p^2+p+1)(p+1)}{(p^2+p+1)p^2}=p+1,\\
N(\mathrm{F},\mathrm{B}) &=\frac{N(\mathrm{B},\mathrm{F}) \cdot |V_{\mathrm{(F)}}|}{|V_{\mathrm{(B)}}|}=\frac{1 \cdot (p^2+p+1)p^2(p+1)}{(p^2+p+1)p^2}=p+1,\\
N(\mathrm{H},\mathrm{B}) &=\frac{N(\mathrm{B},\mathrm{H}) \cdot |V_{\mathrm{(H)}}|}{|V_{\mathrm{(B)}}|}=\frac{1 \cdot \frac{1}{2}(p^3-1)p^3}{(p^2+p+1)p^2}=\tfrac12 p(p-1).
\end{align*}
In total there are $p^5$ matrices in $\C(A)$, so
$$\sum_{\mathrm{X} \in\{\mathrm{A}, \mathrm{B}, \mathrm{C}, \mathrm{E}, \mathrm{F}, \mathrm{H}\}} N(\mathrm{X},\mathrm{B}) \cdot \gen_{\mathrm{(X)}} = p^5.$$
It follows that
\begin{align*}
    N(\mathrm{B},\mathrm{B}) &=\frac{p^5- \sum_{\mathrm{X} \in\{\mathrm{A}, \mathrm{C}, \mathrm{E}, \mathrm{F}, \mathrm{H}\}} N(\mathrm{X},\mathrm{B}) \cdot \gen_{\mathrm{(X)}}}{\gen_{\mathrm{(B)}}}\\
    &=\tfrac{p^5- 1 \cdot p - \tfrac12 (p^2+p) \cdot p(p-1)(p-2)-(p+1) \cdot p(p-1)-(p+1) \cdot p(p-1)^2-\tfrac12 p(p-1) \cdot p^2(p-1)}{p(p-1)}\\
    &=p^2+p+1.
\end{align*}

\medskip 

The results about the neighborhoods of all the vertices of $\Lambda_1(\mm_3(\GF(p)))$ obtained above are summarized in next proposition.

\begin{proposition}
    Let $p$ be a prime and $\mm_3(\GF(p))$ the ring of $3 \times 3$ matrices over $\GF(p)$. The columns of Table~\ref{tab:AllNeighborhoods} describe the neighborhoods of vertices of each type in the unital compressed commuting graph $\Lambda^1(\mm_3(\GF(p)))$. In particular, for every $\mathrm{X},\mathrm{Y} \in \{\mathrm{A}, \mathrm{B}, \ldots, \mathrm{H}\}$, the table gives the number of neighbors of type $(\mathrm{Y})$ in the neighborhood of each vertex of type $(\mathrm{X})$.
\end{proposition}

\begin{table}[h!]
    \caption{Neighborhoods of vertices of $\Lambda_1(\mm_3(\GF(p)))$.}
    \centering
    \begin{tabular}{cc|cccccccc}
             &$\mathrm{(X)}$&
             $\mathrm{(A)}$ & $\mathrm{(B)}$ & $\mathrm{(C)}$ & $\mathrm{(D)}$ & $\mathrm{(E)}$ & $\mathrm{(F)}$ & $\mathrm{(G)}$ & $\mathrm{(H)}$\\
             $\mathrm{(Y)}$&&&&&&&&&\\  \hline 
         $\mathrm{(A)}$ & & 1 & 1 & 1 & 1 & $1$ & 1 & 1 & 1\\
         $\mathrm{(B)}$ & & $(p^2+p+1)p^2$ & $p^2+p+1$ & 3 & 0 & $p^2$ & 1 & 0 & 1\\
         $\mathrm{(C)}$ & & $\frac{1}{6}(p^2+p+1)p^3(p+1)$ & $\frac{1}{2}(p^2+p)$ & 1 & 0 & $0$ & 0 & 0 & 0\\
         $\mathrm{(D)}$ & & $(p^3-1)(p+1)$ & 0 & 0 & 1 & $p-1$ & 0 & 0 & 0\\
         $\mathrm{(E)}$ & & $(p^2+p+1)(p+1)$ & $p+1$ & 0 & 1 & $2p+1$ & 1 & 0 & 0\\
         $\mathrm{(F)}$ & & $(p^2+p+1)p^2(p+1)$ & $p+1$ & 0 & 0 & $p^2$ & 1 & 0 & 0\\
         $\mathrm{(G)}$ & & $\frac{1}{3}(p^3-p)(p^3-p^2)$ & 0 & 0 & 0 & $0$ & 0 & 1 & 0\\
         $\mathrm{(H)}$ & & $\frac{1}{2}(p^3-1)p^3$ & $\frac{1}{2}p(p-1)$ & 0 & 0 & $0$ & 0 & 0 & 1
    \end{tabular}
    \vspace{4mm}
   \label{tab:AllNeighborhoods}
\end{table}

\section{Subgraph induced on 
\texorpdfstring{$V_{\mathrm{(B)}} \cup V_{\mathrm{(E)}}$}{V(B) U V(E)}
}\label{sec:B-E_graph}

In this section we describe the subgraph of $\Lambda^1(\mm_3(\GF(p)))$ induced on the set $V_{\mathrm{(B)}} \cup V_{\mathrm{(E)}}.$ We will refer to this subgraph as $\mathrm{(B)}$--$\mathrm{(E)}$ graph. In order to describe $\mathrm{(B)}$--$\mathrm{(E)}$ graph, we are going to establish a bijective correspondence between the vertices $V_{\mathrm{(B)}} \cup V_{\mathrm{(E)}}$ and the point-line pairs in the projective plane over $\GF(p)$.

\subsection{Projective plane over 
\texorpdfstring{$\GF(p)$}{GF(p)} and its incidence matrix}

    Let us recall the notion of the projective plane over the field $\GF(p),$ which we denote by $\PG(2,p),$ see \cite{Ba97}. Consider the set of all 1-dimensional subspaces of the vector space $\GF(p)^3$ and denote it by $\P$ and the set of all 2-dimensional subspaces of the vector space $\GF(p)^3$ and denote it by $\L$. The points of the projective plane $\PG(2,p)$ are the elements of $\P$ while the lines are the elements of $\L$. Furthermore, a point $P\in \P$ lies on a line $L\in \L$ if and only if $P\subset L$ and we denote this by $P\in L.$
    The cardinalities of these sets are 
    \begin{equation*}
        |\P| = |\L|= p^2+p+1,
    \end{equation*}
    every line contains $p+1$ points and every point lies on $p+1$ lines. 
    The projective plane $\PG(2,p)$ can be described by its incidence matrix which is 0-1 matrix of order $(p^2+p+1)\times (p^2+p+1).$ Each row of the incidence matrix corresponds to a point and each column of the matrix corresponds to a line, where $(P,L)$ entry of the matrix is $1$ if $P\in L$, otherwise it is $0$. Each row and each column of the incidence matrix contains exactly $p+1$ ones.

    The explicit description of the incidence matrix is given in \cite{Ba08}, here we recall and slightly adapt it. Let $e$ be the vector of length $p$ with $1$ at all positions, i.e., \[ e=\begin{bmatrix}
        1 & 1 & \cdots & 1    \end{bmatrix}^{T} \in \mathbb{R}^p.\] 
        For every $s\in \{1,2,\ldots,p\}$ let $R_s \in \mm_p(\mathbb{R})$ be the matrix with $s$-th row equal to $e^T$ and all other entries $0$, i.e., \[{(R_s)}_{i,j}=\begin{cases}
            1, & \text{ if } i=s,\\
            0, & \text{ otherwise. }
        \end{cases} \]
        Furthermore, for every $s\in \{2,\ldots,p\}$ and every $t\in \{1,2,\ldots,p\}$ let $S_{s,t} \in \mm_p(\mathbb{R})$ be the permutation matrix defined as \[{(S_{s,t})}_{i,j}=\begin{cases}
            1, & \text{ if } (s-1)(i+j) \equiv t\ (\text{mod } p),\\
            0, & \text{ otherwise. }
        \end{cases} \]
        Now, let $T_p\in \mm_{1+p+p^2}(\mathbb{R})$ be 0-1 matrix defined as a block matrix 
        \begin{equation}\label{eq:T_p}
        T_p = \begin{bmatrix}
        1 & e^T & 0 & 0 & \cdots & 0 \\
        e & 0 & R_1 & R_2 & \cdots & R_p \\
        0 & R_1^T & I_p & I_p & \cdots & I_p \\
        0 & R_2^T & S_{2,1} & S_{2,2} & \cdots & S_{2,p} \\
        \vdots & \vdots & \vdots & \vdots & \ddots & \vdots \\
        0 & R_p^T & S_{p,1} & S_{p,2} & \cdots & S_{p,p}    \end{bmatrix},  \end{equation} 
        where $I_p$ is the identity matrix of order $p$. It is proved in the \cite{Ba08} that $T_p$ is the incidence matrix of the plane $\PG(2,p).$ 
    
For example, when $p=3$, the incidence matrix $T_3$ of the projective plane $\PG(2,3)$ is
    \[T_3 = \left[ \begin{array}{ c | c c c | c c c | c c c | c c c }
                                1 & 1 & 1 & 1 & 0 & 0 & 0 & 0 & 0 & 0 & 0 & 0 & 0\\ \hline
                                1 & 0 & 0 & 0 & 1 & 1 & 1 & 0 & 0 & 0 & 0 & 0 & 0\\
                                1 & 0 & 0 & 0 & 0 & 0 & 0 & 1 & 1 & 1 & 0 & 0 & 0\\
                                1 & 0 & 0 & 0 & 0 & 0 & 0 & 0 & 0 & 0 & 1 & 1 & 1\\ \hline
                                0 & 1 & 0 & 0 & 1 & 0 & 0 & 1 & 0 & 0 & 1 & 0 & 0\\
                                0 & 1 & 0 & 0 & 0 & 1 & 0 & 0 & 1 & 0 & 0 & 1 & 0\\
                                0 & 1 & 0 & 0 & 0 & 0 & 1 & 0 & 0 & 1 & 0 & 0 & 1\\ \hline
                                0 & 0 & 1 & 0 & 0 & 0 & 1 & 1 & 0 & 0 & 0 & 1 & 0\\
                                0 & 0 & 1 & 0 & 0 & 1 & 0 & 0 & 0 & 1 & 1 & 0 & 0\\
                                0 & 0 & 1 & 0 & 1 & 0 & 0 & 0 & 1 & 0 & 0 & 0 & 1\\\hline
                                0 & 0 & 0 & 1 & 1 & 0 & 0 & 0 & 0 & 1 & 0 & 1 & 0\\
                                0 & 0 & 0 & 1 & 0 & 0 & 1 & 0 & 1 & 0 & 1 & 0 & 0\\
                                0 & 0 & 0 & 1 & 0 & 1 & 0 & 1 & 0 & 0 & 0 & 0 & 1
                   \end{array}\right]. \]

\subsection{Bijection between 
\texorpdfstring{$V_{\mathrm{(B)}} \cup V_{\mathrm{(E)}}$}{V(B) U V(E)} and
\texorpdfstring{$\P \times \L$}{P × L}
}\label{sec:Bijection}

We denote \begin{equation}
        \B = \{ (P,L)\in \P\times \L \mid P\notin L \}
    \end{equation} 
    and 
    \begin{equation}
        \E = \{ (P,L)\in \P\times \L \mid P\in L \}.
    \end{equation} 
    Elements of $\B$ correspond to zeros in the incidence matrix $T_p$ and elements of $\E$ correspond to ones in $T_p$. Obviously, $|\E| = (p+1)(p^2+p+1)$  and $|\B| = p^2(p^2+p+1).$
    
    We first establish a bijection between the sets $V_{\mathrm{(B)}}$ and $\B$. Let $v\in V_{\mathrm{(B)}}$ be an arbitrary vertex. Let $A$ be a generator of vertex $v$, with characteristic polynomial $p_A(x)=(x-\lambda)(x-\mu)^2$, where $\lambda \ne \mu$.
    Denote $\widehat{A}=(\lambda-\mu)^{-1}(A-\mu I)$. Then $\widehat{A}$ is a rank $1$ idempotent such that $v=\rng{A}_1=\rng{\widehat{A}}_1$. Since $\omega_{\widehat{A}}=1$, $\widehat{A}$ is the only rank $1$ idempotent in $v$, so $\widehat{A}$ depends only on $v$ and not on the choice of the generator $A$.
    This means that we can define a mapping $\Phi_{\mathrm{B}} \colon V_{\mathrm{(B)}} \to \B$ by
    \begin{equation*}
        \Phi_{\mathrm{B}}(v)=(\Img \widehat{A}, \Ker \widehat{A}).
    \end{equation*}   
    Note that $L=\Ker \widehat{A} \in \L$, $P=\Img \widehat{A} \in \P$, and $P \notin L$, so that $\Phi_{\mathrm{B}}(v)$ belongs to $\B$.

    In order to show that the mapping $\Phi_{\mathrm{B}}$ is a bijection, we define a mapping $\Psi_{\mathrm{B}}: \B \to V_{\mathrm{(B)}}$ as follows. For a pair $(P,L)\in \B$ we let $\Psi_{\mathrm{B}}(P,L)$ be the subring generated by the unique idempotent in $\mm_3(\GF(p))$ with image $P$ and kernel $L$. Clearly, $\Phi_{\mathrm{B}} \circ \Psi_{\mathrm{B}} = \Id_{\B}$. Since $|V_{\mathrm{(B)}}|=|\B| = p^2(p^2+p+1),$ we conclude that $\Phi_{\mathrm{B}}$ is a bijection.

   Next, we establish a bijection between the sets $V_{\mathrm{(E)}}$ and $\E$. Let $v\in V_{\mathrm{(E)}}$ be an arbitrary vertex. Let $A$ be a generator of vertex $v$ with characteristic polynomial $p_A(x)=(x-\lambda)^3$.
    Denote $\widehat{A}=A-\lambda I$. Then $\widehat{A}$ is a rank $1$ nilpotent such that $v=\rng{A}_1=\rng{\widehat{A}}_1$. Since $\omega_{\widehat{A}}=p-1$, every rank $1$ nilpotent in $v$ is a scalar multiple of $\widehat{A}$, so the kernel and image of $\widehat{A}$ depend only on $v$ and not on the choice of the generator $A$.
   This means that we can define a mapping $\Phi_{\mathrm{E}} \colon V_{\mathrm{(E)}} \to \E$ by
    \begin{equation*}
        \Phi_{\mathrm{E}}(v)=(\Img \widehat{A}, \Ker \widehat{A}).
    \end{equation*}   
    Note that $L=\Ker \widehat{A} \in \L$, $P=\Img \widehat{A} \in \P$, and $P \in L$, so that $\Phi_{\mathrm{E}}(v)$ belongs to $\E$.

    In order to show that the mapping $\Phi_{\mathrm{E}}$ is a bijection, we define a mapping $\Psi_{\mathrm{E}}: \E \to V_{\mathrm{(E)}}$ as follows. For a pair $(P,L)\in \E$ we let $\Psi_{\mathrm{E}}(P,L)$ be the subring generated by a nilpotent in $\mm_3(\GF(p))$ with image $P$ and kernel $L$.
    Since any two nilpotents with this property are scalar multiples of each other, the mapping $\Psi_{\mathrm{E}}$ is well defined. Clearly, $\Phi_{\mathrm{E}} \circ \Psi_{\mathrm{E}} = \Id_{\E}$. Since $|V_{\mathrm{(E)}}|=|\E| = (p^2+p+1)(p+1),$ we conclude that $\Phi_{\mathrm{E}}$ is a bijection.
    
     We now combine the mappings $\Phi_{B}$ and $\Phi_{\mathrm{E}}$ into a bijection $$\Phi: V_{\mathrm{(B)}} \cup V_{\mathrm{(E)}} \to \B \cup \E = \P \times \L$$ defined by $\Phi|_{V_{\mathrm{(B)}}} = \Phi_{\mathrm{B}}$ and  $\Phi|_{V_{\mathrm{(E)}}} = \Phi_{\mathrm{E}}.$ Since $V_{\mathrm{(B)}}$ and $V_{\mathrm{(E)}}$ are disjoint, $\Phi$ is well defined. Since also $\B$ and $\E$ are disjoint and $\Phi_{\mathrm{B}}$ and $\Phi_{\mathrm{E}}$ are bijections, $\Phi$ is a bijection. 

    \subsection{Geometrical interpretation of edges}

    We define a graph $\Delta$ with vertex set $V(\Delta) = \P \times \L$ and edges defined as follows. Let $v_1$ and $v_2$ be elements of $V_{\mathrm{(B)}} \cup V_{\mathrm{(E)}}.$ There is an edge between $\Phi(v_1)$ and $\Phi(v_2)$ in $\Delta$ if and only if there is an edge between $v_1$ and $v_2$ in the unital compressed commuting graph of the ring $\mm_3(\GF(p)).$   
This makes the mapping $\Phi$ a graph isomorphism between the induced subgraph of $\Lambda^1(\mm_3(\GF(p)))$ on the set $V_{\mathrm{(B)}} \cup V_{\mathrm{(E)}}$ and $\Delta.$
We will identify the vertices of the graph $\Delta$ with entries of the incidence matrix $T_p$.
Next theorem describes the edges of the graph $\Delta$ in geometric terms.

\begin{theorem}\label{thm:E(Delta)}
    Let $(P_1,L_1), (P_2,L_2) \in \P \times \L$ be arbitrary. There is an edge between $(P_1,L_1)$ and $(P_2,L_2)$ in $\Delta$ if and only if one of the following conditions holds \begin{enumerate}[$(a)$]
        \item $P_1=P_2$ and $L_1=L_2,$ 
        \item $P_1 \in L_1,$ $P_2\in L_2$, and either $P_1 = P_2$ or $L_1 = L_2,$  
        \item $P_1\neq P_2$, $L_1\neq L_2$ and $P_2\in L_1\cap L_2$ and $P_1\in L_2\setminus L_1,$ 
        \item $P_1\neq P_2$, $L_1\neq L_2$ and $P_1\in L_1\cap L_2$ and $P_2\in L_1\setminus L_2,$ 
        \item $P_1\neq P_2$, $L_1\neq L_2$ and $P_1\in L_2\setminus L_1$ and $P_2\in L_1\setminus L_2.$  
    \end{enumerate}
\end{theorem}
\begin{proof}
Let $(P_1,L_1)=\Phi(v_1)$ and $(P_2,L_2)=\Phi(v_2)$. Let $A_1$ be the generator of $v_1$ with image $P_1$ and kernel $L_1$, and $A_2$ be the generator of $v_2$ with image $P_2$ and kernel $L_2$.

\noindent $(\Leftarrow)$: Suppose that one of the conditions $(a)$--$(e)$ holds. 
    \begin{enumerate}[$(a)$]
        \item As $(P_1,L_1)=(P_2,L_2)$ and $\Phi$ is bijection then $v_1=v_2.$ Since every vertex in $\Lambda^1(\mm_3(\GF(p)))$ has a loop there is an edge between $(P_1,L_1)$ and $(P_2,L_2).$
        \item The conditions imply that $(P_1,L_1), (P_2,L_2)\in \E,$ so that $v_1,v_2\in V_{\mathrm{(E)}}.$  
        
        If $P_1=P_2$ then $P_2\in L_1$ and $P_1\in L_2,$ so $A_1A_2=0$ and $A_2A_1=0$. Combining last two equation we get $A_1A_2=A_2A_1$ which means that there is an edge between $v_1$ and $v_2$, hence also between $(P_1,L_1)$ and $(P_2,L_2).$

        If $L_1 = L_2$ then again $P_2\in L_1$ and $P_1\in L_2,$ and we obtain the same conclusion.  
        \item[$(c)$--$(e)$] Each set of the conditions imply that $P_2\in L_1$ and $P_1\in L_2,$ so there is an edge between $(P_1,L_1)$ and $(P_2,L_2),$ as shown above. 
    \end{enumerate}
    
\noindent$(\Rightarrow)$: Suppose that there is an edge between $(P_1,L_1)$ and $(P_2,L_2)$ in $\Delta.$ Then there is an edge between $v_1$ and $v_2$ so we have $A_1A_2=A_2A_1.$ If $(P_1,L_1)=(P_2,L_2)$ then condition $(a)$ holds. So suppose that $(P_1,L_1)\neq(P_2,L_2)$. Note that matrices $A_1$ and $A_2$ are of rank 1, as their images are $P_1$ and $P_2$, which are vector subspaces of dimension 1. 

We claim that $A_1A_2=0.$ Suppose this is not the case. Then $\rank (A_1A_2) = 1.$ This implies that $\Img(A_1A_2)=\Img(A_1)=P_1$ and $\Ker (A_1A_2) = \Ker(A_2) = L_2.$ As $A_1A_2=A_2A_1$, we conclude similarly $\Img(A_2A_1)=\Img(A_2)=P_2$ and $\Ker (A_2A_1) = \Ker(A_1) = L_1.$ Hence, $P_1=P_2$ and $L_1=L_2$ which is in contradiction with $(P_1,L_1)\neq(P_2,L_2)$. This proves our claim.

From the above we get $A_1A_2=0$ and $A_2A_1=0.$ This implies $P_2\in L_1$ and $P_1\in L_2.$ We now consider four cases:
\begin{enumerate}[$(i)$]
    \item $P_1\in L_1$ and $P_2\in L_2$: If $P_1=P_2$ or $L_1=L_2$ then the condition $(b)$ holds. So, suppose the opposite $P_1\neq P_2$ and $L_1\neq L_2.$ This means that two different points $P_1$ and $P_2$ lie at the same time on two different lines $L_1$ and $L_2$, which is impossible.
    \item $P_1\notin L_1$ and $P_2\in L_2$: Since $P_1\notin L_1$ and $P_2\in L_1$ we have $P_1\neq P_2.$ Similarly, as $P_1\notin L_1$ and $P_1\in L_2$ we get $L_1\neq L_2.$ Furthermore, $P_2\in L_1 \cap L_2$ and $P_1\in L_2\setminus L_1.$ Hence, condition $(c)$ holds.
    \item $P_1\in L_1$ and $P_2\notin L_2$: Since $P_2\notin L_2$ and $P_1\in L_2$ we have $P_1\neq P_2.$ Similarly, as $P_2\notin L_2$ and $P_2\in L_1$ we get $L_1\neq L_2.$ Furthermore, $P_1\in L_1 \cap L_2$ and $P_2\in L_1\setminus L_2.$ Hence, condition $(d)$ holds.
    \item $P_1\notin L_1$ and $P_2\notin L_2$: Since $P_1\in L_2$ and $P_1\notin L_1$ we get $P_1\in L_2\setminus L_1.$ Similarly, as $P_2\in L_1$ and $P_2\notin L_2$ we get $P_2\in L_1 \setminus L_2$. This implies that $P_1 \neq P_2$ and $L_1 \neq L_2,$ hence condition $(e)$ holds.
\end{enumerate}
This finishes the proof.
\end{proof}

We remark that condition $(a)$ of Theorem~\ref{thm:E(Delta)} describes the loops in $\Delta.$ Condition $(b)$ describes the edges between two different vertices in $\E$, it means that two ones in the incidence matrix are connected if and only if they lie in the same row or in the same column. Conditions $(c)$ and $(d)$ describe edges between vertices in $\B$ and $\E$, they mean that 0 and 1 in the incidence matrix are connected if and only if they lie in different rows and columns, and the other two entries of the $2\times 2$ submatrix of $T_p,$ which contains the two entries 0 and 1, are both equal to 1. Condition $(e)$ describes edges between two different vertices in $\B$. It means that two zeros in the incidence matrix $T_p$ are connected if and only if they lie in different rows and columns, and the other two entries of the $2\times 2$ submatrix of $T_p,$ which contains the two zeros, are both equal to 1.  

Figure \ref{fig:submatrices} shows all possible $2\times 2$ submatrices of $T_p$ and the edges between their entries. The red vertices in the figure are vertices of type $\mathrm{(B)}$ and correspond to the zeroes in matrix $T_p,$ while the blue vertices are vertices of type $\mathrm{(E)}$ and correspond to ones in matrix $T_p.$ Note that a $2\times 2$ submatrix of $T_p$ cannot contain only ones because that would mean that two different lines intersect in two different points, which is not possible in the projective plane.

\begin{figure}
    \centering
     \includegraphics[width=0.8\textwidth]{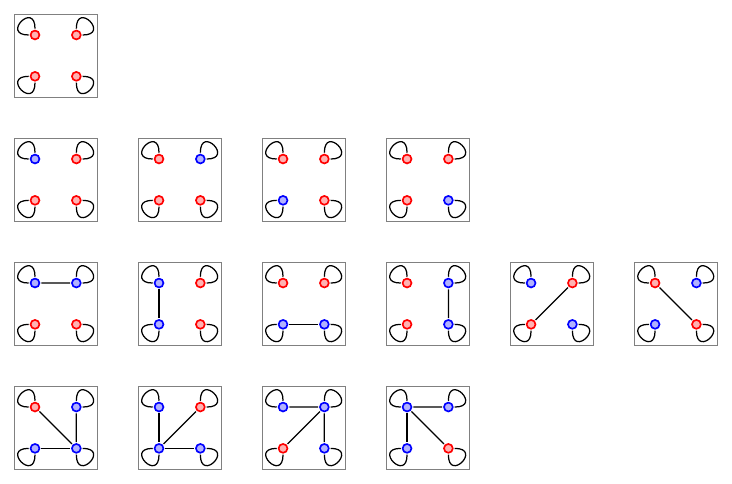}
    \caption{Possible $2\times 2$ submatrices of $T_p$ and the edges between their entries.}\label{fig:submatrices}
\end{figure}

\section{Description of 
\texorpdfstring{$\Lambda^1(\mathcal{M}_3(\GF(p)))$}{Λ¹(ℳ₃(\GF(p)))}
}\label{sec:Description}

In this section we give the  complete description of the unital compressed commuting graph $\Lambda^1(\mathcal{M}_3(\GF(p)))$ of the ring of $3 \times 3$ matrices over $\GF(p)$ and give an algorithm for its construction.

Note that the properties of the subgraph of $\Lambda^1(\mathcal{M}_3(\GF(p)))$ induced on the set of vertices $V_{\mathrm{(B)}} \cup V_{\mathrm{(E)}}$ are described in Section~\ref{sec:B-E_graph}. This subgraph can be constructed by considering entries of the incidence matrix $T_p$ as vertices and setting the edges according to Figure~\ref{fig:submatrices}. Here we describe the properties of the rest of the graph, according to the type of vertices, explaining how the remaining vertices are attached to the $\mathrm{(B)}$--$\mathrm{(E)}$ subgraph. Furthermore, according to Table~\ref{tab:AllNeighborhoods}, the subgraph of $\Lambda^1(\mathcal{M}_3(\GF(p)))$ induced on the set of vertices $V_{\mathrm{(C)}} \cup V_{\mathrm{(D)}} \cup V_{\mathrm{(F)}} \cup V_{\mathrm{(G)}} \cup V_{\mathrm{(H)}}$ contains only loops and no other edges. These properties will allow us to give an algorithm for the construction of the whole graph.

\begin{enumerate}[(A)]
    \setcounter{enumi}{2}
    \item Note that for $p=2$ there are no vertices of type $\mathrm{(C)}$, so suppose that $p\geq 3.$
    By Table~\ref{tab:AllNeighborhoods}, every vertex of type $\mathrm{(C)}$ is connected to $3$ vertices of type $\mathrm{(B)}$, which form a triangle, since they are subrings of a commutative ring, namely, the vertex of type $(\mathrm{C})$. So, every vertex of type $\mathrm{(C)}$ is connected to the vertices of a unique triangle, whose vertices are of type $\mathrm{(B)}$ in the $\mathrm{(B)}$--$\mathrm{(E)}$ graph.
    Now, we show that for every triangle of vertices of type $\mathrm{(B)}$ there is a vertex of type $\mathrm{(C)}$ connected to them.
    Choose three different vertices of type $(\mathrm{B})$ that form a triangle. Since their generators commute, they are simultaneously diagonalizable, so they commute with any matrix of type $\mathrm{(C)}$ that is diagonalizable in the same basis. All such matrices of type $(\mathrm{C})$ generate the same vertex of type $(\mathrm{C})$. Note that this vertex is unique because the basis is unique up to scalar multiplication of its members.
    As a consequence, the number of triangles of vertices of type $\mathrm{(B)}$ is equal to the number of vertices of type $\mathrm{(C)}$, which is equal to $\frac{1}{6}(p^2+p+1)p^3(p+1)$, by Table~\ref{tab:RedTable}.
    
    \setcounter{enumi}{5}
    \item By Table~\ref{tab:AllNeighborhoods}, every vertex of type $\mathrm{(F)}$ is connected to precisely one vertex of type $\mathrm{(B)}$ and one vertex of type $\mathrm{(E)}$. These two vertices are connected by an edge since they are subrings of the same commutative ring. So, every vertex of type $\mathrm{(F)}$ is connected to the endpoints of a unique $\mathrm{(B)}$--$\mathrm{(E)}$ edge.
    From Table~\ref{tab:RedTable} it follows that the number of vertices of type $\mathrm{(B)}$ is $(p^2+p+1)p^2$ and from Table~\ref{tab:AllNeighborhoods} we have that each vertex of type $\mathrm{(B)}$ is connected to $p+1$ vertices of type $\mathrm{(E)}$, so the number of $\mathrm{(B)}$--$\mathrm{(E)}$ edges is equal to $(p^2+p+1)p^2\cdot (p+1)$. Furthermore, this is equal to the number of vertices of type $\mathrm{(F)}$. Hence, the endpoints of every $\mathrm{(B)}$--$\mathrm{(E)}$ edge are connected to a unique vertex of type $\mathrm{(F)}$.
    
    \setcounter{enumi}{7}
    \item By Table~\ref{tab:AllNeighborhoods}, every vertex of type $\mathrm{(H)}$ is connected to a unique vertex of type $\mathrm{(B)}$. Every vertex of type $\mathrm{(B)}$ has $\tfrac{p(p-1)}{2}$ vertices of type $\mathrm{(H)}$ in the neighborhood.
    From Table~\ref{tab:RedTable} we see that there are $\frac{1}{2}(p^3-1)p^3$ vertices of type $\mathrm{(H)}$ and $(p^2+p+1)p^2$ vertices of type $\mathrm{(B)}$.
    Note that $\frac{1}{2}(p^3-1)p^3=(p^2+p+1)p^2 \cdot \tfrac{p(p-1)}{2}$, hence, vertices of type $\mathrm{(H)}$ are partitioned into $(p^2+p+1)p^2$ groups, and each group is connected to a unique vertex of type $\mathrm{(B)}$.
    
    \setcounter{enumi}{3}
    \item By Table~\ref{tab:AllNeighborhoods}, every vertex of type $\mathrm{(D)}$ is connected to a unique vertex of type $\mathrm{(E)}$. Every vertex of type $\mathrm{(E)}$ has $p-1$ vertices of type $\mathrm{(D)}$ in the neighborhood.
    There are $(p^3-1)(p+1)$ vertices of type $\mathrm{(D)}$ and $(p^2+p+1)(p+1)$ vertices of type $\mathrm{(E)}$.
    Since $(p^3-1)(p+1)=(p^2+p+1)p^2 \cdot (p-1)$, vertices of type $\mathrm{(D)}$ are partitioned into $p-1$ groups, and each group is connected to a unique vertex of type $\mathrm{(E)}$.
    
    \setcounter{enumi}{6}
    \item By Table~\ref{tab:RedTable}, there are $\frac{1}{3}(p^3-p)(p^3-p^2)$ vertices of type $\mathrm{(G)}$ and they are not connected to the $\mathrm{(B)}$--$\mathrm{(E)}$ graph.
    
    \setcounter{enumi}{0}
    \item There is a unique vertex of type $\mathrm{(A)}$ and it its connected to every vertex of the graph. 
\end{enumerate}

We are now ready to give an algorithm for the construction of the unital compressed commuting graph $\Lambda^1(\mathcal{M}_3(\GF(p)))$ of the ring $\mm_3(\GF(p))$.

\begin{enumerate}[1.]
    \item We construct the $\mathrm{(B)}$--$\mathrm{(E)}$ graph as described in Section~\ref{sec:B-E_graph}. The vertices of type $\mathrm{(B)}$ correspond to the zeroes in the incidence matrix $T_p$ of the projective geometry $\PG(2,p)$ given in \eqref{eq:T_p}, and the vertices of type $\mathrm{(E)}$ correspond to ones in the same matrix. The edges between these vertices are presented in Figure~\ref{fig:submatrices}.
    \item If $p\geq 3$ then for every triangle of vertices of type $\mathrm{(B)}$ we add one vertex of type $\mathrm{(C)}$ and connect it the vertices of the triangle. If $p=2$ we skip this step.
    \item For every $\mathrm{(B)}$--$\mathrm{(E)}$ edge we add one vertex of type $\mathrm{(F)}$ and connect it to the endpoints of the edge.
    \item For every vertex of type $\mathrm{(B)}$ we add $\frac{p(p-1)}{2}$ vertices of type $\mathrm{(H)}$ and connect them to the vertex of type $\mathrm{(B)}$.
    \item For every vertex of type $\mathrm{(E)}$ we add $p-1$ vertices of type $\mathrm{(D)}$ and connect them to the vertex of type $\mathrm{(E)}$.
    \item We add $\frac{1}{3}(p^3-p)(p^3-p^2)$ vertices of type $\mathrm{(G)}$.
    \item We add one vertex of type $\mathrm{(A)}$ and connect it to every other vertex.
    \item We put a loop on every vertex.
\end{enumerate}

The above algorithm gives the complete description of the unital compressed commuting graph $\Lambda^1(\mm_3(\GF(p)))$, which is the main contribution of this paper.

\section{Commuting graph of \texorpdfstring{$\mm_3(\GF(p))$}{ℳ₃(\GF(p))}}\label{sec:CG}

Finally, we demonstrate how the unital compressed commuting graph $\Lambda^1(\mathcal{M}_3(\GF(p)))$ can be used to describe the ordinary (non-compressed) commuting graph $\Gamma(\mathcal{M}_3(\GF(p)))$.

For $2\times 2$ matrices over a finite field $\mathbb{F}$ the structure of the commuting graph $\Gamma(\mm_2(\mathbb{F}))$ is described in \cite[Theorem 2]{AkGhHaMo04}. In particular, this graph is a disjoint union of $|\mathbb{F}|^2+|\mathbb{F}|+1$ cliques of size $|\mathbb{F}|^2-|\mathbb{F}|.$  For $3 \times 3$ matrices the description of the commuting graph $\Gamma(\mm_3(\mathbb{F}))$ is still an open problem. The graph was partially described in \cite[Lemma 4.1]{DoKoKu17} where the authors showed that the graph is not connected and all but one of its components are cliques. Furthermore, every connected component that is a clique equals $\mathbb{F}[A]\setminus \mathbb{F}I$ where $A$ is a non-derogatory matrix with irreducible minimal polynomial such that there is no intermediate field between fields $\mathbb{F}$ and $\mathbb{F}[A].$

Using the results of Section~\ref{sec:Description} we can now completely describe the graph $\Gamma(\mm_3(\mathbb{F}))$ in the case when $\mathbb{F} = \GF(p)$. We will do this using the so-called "blow-up" process that was originally used for zero-divisor graphs in \cite{DuJeSt21a, BlWi10}.

We start with $\Lambda^1(\mm_3(\GF(p)))$ described in Section~\ref{sec:Description}. To obtain the graph $\Gamma(\mm_3(\GF(p)))$ we first remove the unique vertex of type $\mathrm{(A)}$ and all edges incident to this vertex. Then we remove the loop  from every vertex. In the final step, we "blow-up" (one by one in an arbitrary order) each vertex $v$ of type $\mathrm{(X)}$ into a clique of size $\gen_{\mathrm{(X)}}$, where $\gen_{\mathrm{(X)}}$ is given in Table~\ref{tab:RedTable}. Furthermore, we connect each vertex of this clique to every vertex that was connected to $v$. Once we do this for all the vertices, we obtain the graph $\Gamma(\mm_3(\GF(p))).$ 

We remark that after removing the unique vertex of type $\mathrm{(A)}$ along with all of his edges and all the loops from $\Lambda^1(\mm_3(\GF(p))),$ the graph breaks into several connected components. All but one of the components are single vertices, these are precisely vertices of type $\mathrm{(G)}$ and there are $\frac{1}{3}(p^3-p)(p^3-p^2)$ of them. After the "blow-up" process these become cliques of size $p^3-p$ in $\Gamma(\mm_3(\GF(p))).$ The only remaining connected component contains all the matrices of types $\mathrm{(B)}$, $\mathrm{(C)}$, $\mathrm{(D)}$, $\mathrm{(E)}$, $\mathrm{(F)}$, and $\mathrm{(H)}$. This is in accordance with the partial description in \cite[Lemma 4.1]{DoKoKu17}.

\section*{Acknowledgments}
The support by the bilateral grants BI-BA/24-25-024 and BI-BA/26-27-008 of the ARIS (Slovenian Research and Innovation Agency) is gratefully acknowledged. Damjana Kokol Bukovšek and Nik Stopar acknowledge financial support from the ARIS (research core funding No. P1-0222 and projects J1-70034 and J1-50002). Ivan-Vanja Boroja acknowledges financial support (No. 01-67-6/21) from the Municipality of Mrkonjić Grad.

\bibliographystyle{amsplain}
\bibliography{biblio}

\end{document}